\documentclass[10pt]{article}
\usepackage{amsmath}
\usepackage{amsthm}
\usepackage{amssymb}
\usepackage{amsfonts}

\newtheorem{theorem}{Theorem}[section]
\newtheorem{lemma}[theorem]{Lemma}

\begin{document}

\begin{center}
{\large\bf On a method of evaluation of zeta-constants based on one number theoretic approach
}\end{center}
\vspace*{-3mm}
\begin{center}
{\large\rm\bf\copyright\,\,2017\,\,\,\,\,\,\,
E.A. Karatsuba$^{*}$}\\
\medskip
{\it\,\, (*119333 \, Moscow, \,\,Vavilova str., 40, 
Research Center ''Computer Science and Control'' of Russian Academy of Sciences)\\
e--mail: ekar@ccas.ru, karacuba@mi.ras.ru}
\end{center}






----------------------------------------------------------------------

\begin{abstract}
New formulas for approximation of zeta-constants were derived on the basis of a number-theoretic approach constructed for the irrationality proof of certain classical constants. Using these formulas it's possible to approximate certain zeta-constants and their combinations by rational fractions and construct a method for their evaluation.
\end{abstract}

\section{Introduction}
The main purpose of this paper is to derive new formulas for approximation of zeta-constants, that is the values of the Riemann zeta function  $\zeta(n)$ with $n\geq 2$, $n$ -- integers, on the basis of a method which is used in the irrationality proofs.

The problem of constructing the methods for approximation and calculation of values of the Riemann zeta function at points of different arithmetic nature, and especially of zeta-constants,  has been considered by many authors (we recall the works: \cite{1}--\cite{9}).
In \cite{10} a new method for approximation of these constants by rational fractions was presented. This method is constructed on the basis of the Hermite-Beukers approach (see \cite{11}-- \cite{13}), which the latter one applied to prove the irrationality of zeta-constants $\zeta(2)$ and $\zeta(3)$ using two especially selected polynomials $P_n(x)$ and $Q_n(x),$ $n\geq 1:$

\begin{equation}\label{eq:0-1}
P_n(x) = \frac{1}{n!}\left(\frac{d}{dx}\right)^n\left(x^n(1-x)^n\right), \  \  \
Q_n(x) = (1-x)^n. 
\end{equation}
Various versions of the Beukers method, its modifications and generalizations are widely used in the study of values of various functions on the irrationality (see, for example, \cite{14}--\cite{19}).

In studying the values of zeta-constants on the basis of Hermite-Beukers' approach and in constructing rational approximations to them the problem arises to derive explicit formulas with coefficients of polynomials participating in integrals. For the case of two polynomials first such explicit formulas with coefficients in the combinations with zeta-constants depending on coefficients of polynomials in the most general canonical form were derived in \cite{10}. The method from  \cite{10} gives a possibility to approximate zeta-constants and some of their combinations by enough simple expressions from rational fractions with coefficients of the polynomials (\ref{eq:0-1}) and calculate them effectively. However for the study of zeta-constants $\zeta(n)$, $n\geq 2$, in full volume, it's necessary to be able to derive such formulas for the case of $n$, $n\geq 2$, polynomials.

The present paper continues the study begun in \cite{10}. 
New explicit formulas for approximation of the values of zeta-constants with participation of three polynomials are obtained. 
The possibility to get such formulas provide combinatorial lemmas \ref{lm:1} and \ref{lm:2}.
An algorithm for calculation of zeta-constants based on new formulas is described. One should mention that method from \cite{10} assumes that both polynomials are needed to provide a reasonably good approximation (like in the Beukers method, see \cite{12}, \cite{13}). The present method with three polynomials opens up a new choice opportunity between two ways:  1) all three polynomials will provide a convergence rate;
2) two polynomials  will provide a convergence rate, and by manipulating the coefficients of the third polynomial, one can reduce the number of calculated rational fractions. 
Taking into account such a possibility, we will consider the third polynomial in the most general canonical form throughout the paper.

Further we use the following standard notations:
a generalized harmonic number $H_n^{(m)}$ of order $m$ is the following sum for positive integers $n$ and $m$:

\begin{equation}\label{eq:1-1}
H_n^{(m)} = \sum_{k=1}^n\frac{1}{k^m},  \  \  H_n^{(1)} = H_n, \  \   H_0^{(m)} =0.
\end{equation}
The properties and methods for evaluation of harmonic numbers are widely studied: see, for example, \cite{20}, \cite{21}.

\section{Auxiliary lemmas}

For $|t| < 1$:

$$
\frac{1}{1-t} = 1 + t + t^2 + \dots = \sum_{k=0}^{\infty} t^k.
$$
Hence

$$
I_3 = \int_0^1\int_0^1\int_0^1\frac{dx dy dz}{1 -xyz} =
\int_0^1\int_0^1\int_0^1\sum_{k=0}^{\infty} x^k y^k z^k dx dy dz =
$$
%
$$
\sum_{k=0}^{\infty}\left(\int_0^1 x^k dx\right)\left(\int_0^1 y^k dy\right)\left(\int_0^1 z^k dz\right) =
\sum_{k=0}^{\infty}\frac{1}{(k+1)^3} = \zeta(3).
$$
%

\begin{lemma}\label{lm:1}
For any $r_1, r_2, r_3 \geq 0;$ $s\geq 3;$  the following relation holds:

$$
I(r_1, r_2, r_3) = \int_0^1\dots\int_0^1\frac{x_1^{r_1}x_2^{r_2}x_3^{r_3}}{1 -x_1x_2x_3\dots x_s}dx_1dx_2dx_3\dots dx_s =
$$
\begin{equation}\label{eq:1-2}
= \sum_{k=0}^{\infty}\frac{1}{(r_1 +k+1)(r_2 +k+1)(r_3 +k+1)(k+1)^{s-3}}.
\end{equation}
\end{lemma}
\begin{proof} For $0< x_1, \dots, x_s <1;$ $0< x_1\cdot x_2\cdot\ldots\cdot x_s <1;$
we have:

$$
\frac{1}{1-x_1\cdot\ldots\cdot x_s} = \sum_{k=0}^{\infty}(x_1\cdot\ldots\cdot x_s)^k,
$$

$$
I(r_1, r_2, r_3) = \int_0^1\dots\int_0^1\frac{x_1^{r_1}x_2^{r_2}x_3^{r_3}}{1 -x_1x_2x_3\dots x_s}dx_1dx_2dx_3\dots dx_s 
$$
$$
=\sum_{k=0}^{\infty}\int_0^1 x_1^{k+r_1}d x_1\int_0^1 x_2^{k+r_2}d x_2\int_0^1 x_3^{k+r_3}d x_3\int_0^1 x_4^{k}d x_4\dots\int_0^1 x_s^{k}d x_s 
$$
$$
= \sum_{k=0}^{\infty}\frac{1}{r_1 +k+1}\cdot\frac{1}{r_2 +k+1}\cdot\frac{1}{r_3 +k+1}\cdot\frac{1}{(k+1)^{s-3}}.
$$
%
\end{proof}

\begin{lemma}\label{lm:2}
For any integer $s\geq 1$ and for any $r\geq 1;$ \ $k\geq 0$, the following identities hold:

\begin{equation}\label{eq:1-3}
\frac{1}{(r+k+1)(k+1)^s} = \sum_{j=1}^{s}\frac{(-1)^{j-1}}{r^j(k+1)^{s+1-j}} + \frac{(-1)^s}{r^s(r+k+1)},
\end{equation}
$$
\frac{1}{(r+k+1)^2(k+1)^s} = \sum_{j=1}^{s}\frac{(-1)^{j-1}j}{r^{j+1}(k+1)^{s+1-j}}
+ \frac{(-1)^ss}{r^{s+1}(r+k+1)} 
$$
\begin{equation}\label{eq:1-4}
 + \frac{(-1)^s}{r^{s}(r+k+1)^2},
\end{equation}
$$
\frac{1}{(r+k+1)^3(k+1)^s} = \sum_{j=1}^{s}\frac{(-1)^{j-1}j(j+1)}{2 r^{j+2}(k+1)^{s+1-j}} + \frac{(-1)^{s}s(s+1)}{2 r^{s+2}(r+k+1)} 
$$
\begin{equation}\label{eq:1-5}
+ \frac{(-1)^{s}s}{r^{s+1}(r+k+1)^2} + \frac{(-1)^{s}}{r^{s}(r+k+1)^3}.
\end{equation}
\end{lemma}
\begin{proof} 
The formulas (\ref{eq:1-3})-- (\ref{eq:1-5}) are proved by induction on $s$. Let's prove, for example,  (\ref{eq:1-5}).
For $s=1$ it's possible to verify directly that the equality is true

$$
\frac{1}{(r+k+1)^3(k+1)} =$$
\begin{equation}\label{eq:1-6}
\frac{1}{r^3(k+1)} - \frac{1}{r^3(r+k+1)} -\frac{1}{r^2(r+k+1)^2}- \frac{1}{r(r+k+1)^3}
\end{equation}
Let
(\ref{eq:1-5}) be true for  $s\leq n,$ and

$$
\frac{1}{(r+k+1)^3(k+1)^n} = \sum_{j=1}^{n}(-1)^{j-1}\frac{j(j+1)}{2r^{j+2}(k+1)^{n+1-j}} + (-1)^{n}\frac{n(n+1)}{2r^{n+2}(r+k+1)} 
$$
\begin{equation}\label{eq:1-7}
+ (-1)^{n}\frac{n}{r^{n+1}(r+k+1)^2} + (-1)^{n}\frac{1}{r^{n}(r+k+1)^3}.
\end{equation}
Let's prove that (\ref{eq:1-5}) is true also for   $s = n+1$. We have from (\ref{eq:1-7})

$$
\frac{1}{(r+k+1)^3(k+1)^{n+1}} = \frac{1}{k+1}\sum_{j=1}^{n}
\frac{(-1)^{j-1}j(j+1)}{2r^{j+2}(k+1)^{n+1-j}} +
$$
%
$$
\frac{(-1)^{n}n(n+1)}{2r^{n+2}(k+1)(r+k+1)}+ \frac{(-1)^{n}n}{r^{n+1}(k+1)(r+k+1)^2} + \frac{(-1)^{n}}{r^{n}(k+1)(r+k+1)^3}.
$$
%
Substituting in 
the last expression  (\ref{eq:1-3}), (\ref{eq:1-4}) for $s= 1$ and  (\ref{eq:1-6}) we find

$$\frac{1}{(r+k+1)^3(k+1)^{n+1}} =
\sum_{j=1}^{n}\frac{(-1)^{j-1}j(j+1)}{2 r^{j+2}(k+1)^{n+2-j}} $$
$$+
\frac{(-1)^{n}n(n+1)}{2 r^{n+2}}\left(\frac{1}{r(k+1)}-\frac{1}{r(r+k+1)}\right)
$$

$$
+ \frac{(-1)^{n}n}{r^{n+1}}\left(\frac{1}{r^2(k+1)}-\frac{1}{r^2(r+k+1)}-\frac{1}{r(r+k+1)^2}\right)
$$

$$
+ \frac{(-1)^{n}}{r^{n}}\left(\frac{1}{r^3(k+1)}-\frac{1}{r^3(r+k+1)}-\frac{1}{r^2(r+k+1)^2}-\frac{1}{r(r+k+1)^3}\right)
$$

$$
= \sum_{j=1}^{n}\frac{(-1)^{j-1}j(j+1)}{2 r^{j+2}(k+1)^{n+2-j}} +
\frac{(-1)^{n}\left(\frac{n(n+1)}{2}+n+1\right)}{r^{n+3}(k+1)}
$$

$$
+ \frac{(-1)^{n+1}\left(\frac{n(n+1)}{2}+n+1\right)}{r^{n+3}(r+k+1)} +
\frac{(-1)^{n+1}(n+1)}{r^{n+2}(r+k+1)^2}
+\frac{(-1)^{n+1}}{r^{n+1}(r+k+1)^3}
$$

$$
= \sum_{j=1}^{n+1}\frac{(-1)^{j-1}j(j+1)}{2 r^{j+2}(k+1)^{n+2-j}} +
\frac{(-1)^{n+1}(n+1)(n+2)}{2 r^{n+3}(r+k+1)}$$
$$
+ \frac{(-1)^{n+1}(n+1)}{r^{n+2}(r+k+1)^2} +
\frac{(-1)^{n+1}}{r^{n+1}(r+k+1)^3},
$$
that is the formula (\ref{eq:1-5}) for $s= n+1.$
\end{proof}

\section{Main formula}

Let $a_0, a_1, \dots, a_n, \  b_0, b_1, \dots, b_n, \ c_0, c_1, \dots, c_n$, $n\geq 1$, be arbitrary numbers, and let 
$P_n(x)$, $Q_n(x)$ and $T_n(x)$ be the polynomials:

\begin{align*}
P_n(x) = a_0 + a_1 x + \dots + a_n x^n,\\
Q_n(x) = b_0 + b_1 x + \dots + b_n x^n,\\
T_n(x) = c_0 + c_1 x + \dots + c_n x^n.
\end{align*}
Multiply the left and right sides of the relation
(\ref{eq:1-2}) to $a_{r_1}b_{r_2}c_{r_3}$ and sum over  $r_1,
r_2, r_3 = 0,1,\dots, n; s\geq 3.$ We get

$$
I_s = I_s(n) = \sum_{r_1=0}^{n}\sum_{r_2=0}^{n}\sum_{r_3=0}^{n}a_{r_1}b_{r_2}c_{r_3} I(r_1,r_2,r_3)$$
$$ = \int_0^1\dots\int_0^1\frac{P_n(x_1)Q_n(x_2)T_n(x_3)}{1-x_1x_2x_3\dots x_s}dx_1 dx_2dx_3\dots dx_s 
$$
\begin{equation}\label{eq:1-10}
= \sum_{k=0}^{\infty}\frac{1}{(k+1)^{s-3}}\sum_{r_1=0}^{n}\sum_{r_2=0}^{n}
\sum_{r_3=0}^{n}\frac{a_{r_1}b_{r_2}c_{r_3}}{(r_1+k+1)(r_2+k+1)(r_3+k+1)}.
\end{equation}
Separating in (\ref{eq:1-10}) the terms 
with $r_1= r_2=r_3 =0,$ then with  $r_1=r_2=0$, 
$r_3\ne 0$;
 $r_2=r_3 =0$, $r_1\ne 0$; $r_1=r_3 =0$, $r_2\ne 0$, 
and finally with  $r_1=0$, $r_2\ne 0, r_3\ne 0$; $r_2=0$, $r_1\ne 0, r_3\ne 0$; $r_3 =0,$ $r_1\ne 0, r_2\ne 0$,
we find
$$
I_s = a_0b_0c_0\zeta(s) + a_0b_0\sum_{k=0}^{\infty}\frac{1}{(k+1)^{s-1}}\sum_{r_3=1}^{n} \frac{c_{r_3}}{r_3+k+1}$$
$$ +
a_0c_0\sum_{k=0}^{\infty}\frac{1}{(k+1)^{s-1}}\sum_{r_2=1}^{n}\frac{b_{r_2}}{r_2+k+1}
 + b_0c_0\sum_{k=0}^{\infty}\frac{1}{(k+1)^{s-1}}\sum_{r_1=1}^{n}\frac{a_{r_1}}{r_1+k+1}$$
 $$
+ a_0\sum_{k=0}^{\infty}\frac{1}{(k+1)^{s-2}}\sum_{r_2=1}^{n}\sum_{r_3=1}^{n}\frac{b_{r_2}c_{r_3}}{(r_2+k+1)(r_3+k+1)}$$ 
$$+ b_0\sum_{k=0}^{\infty}\frac{1}{(k+1)^{s-2}}\sum_{r_1=1}^{n}\sum_{r_3=1}^{n}\frac{a_{r_1}c_{r_3}}{(r_1+k+1)(r_3+k+1)}$$
$$+ c_0\sum_{k=0}^{\infty}\frac{1}{(k+1)^{s-2}}\sum_{r_1=1}^{n}\sum_{r_2=1}^{n}\frac{a_{r_1}b_{r_2}}{(r_1+k+1)(r_2+k+1)}
$$
\begin{equation}\label{eq:1-11}+
\sum_{k=0}^{\infty}\frac{1}{(k+1)^{s-3}}\sum_{r_1=1}^{n}\sum_{r_2=1}^{n}\sum_{r_3=1}^{n}
\frac{a_{r_1}b_{r_2}c_{r_3}}{(r_1+k+1)(r_2+k+1)(r_3+k+1)}.
\end{equation}
We divide the double sums in $r_{\xi},$ $\xi= 1, 2, 3;$ on the right side (\ref{eq:1-11}) into two sums:
the sum in which the summation indices are equal and the sum in which they are not equal, that is, for example,

$$
\sum_{r_1=1}^{n}\sum_{r_2=1}^{n}\frac{a_{r_1}b_{r_2}}{(r_1+k+1)(r_2+k+1)}$$
%
\begin{equation}\label{eq:1-12}
=
\sum_{r_1=1}^{n}\frac{a_{r_1}b_{r_1}}{(r_1+k+1)^2}+
\sum_{\substack{{r_1,r_2=1}\\r_1\ne r_2}}^{n}\frac{a_{r_1}b_{r_2}}{r_2-r_1}\left(\frac{1}{r_1+k+1}-\frac{1}{r_2+k+1}\right).
\end{equation}
In a triple sum over $r_{\xi}$ from (\ref{eq:1-11}), we also separate terms with the same indices, using the relations

$$
\sum_{r_1=1}^{n}\sum_{r_2=1}^{n}\sum_{r_3=1}^{n}
\frac{a_{r_1}b_{r_2}c_{r_3}}{(r_1+k+1)(r_2+k+1)(r_3+k+1)}  = \sum_{r_1=1}^{n}\frac{a_{r_1}b_{r_1}c_{r_1}}{(r_1+k+1)^3}
$$
$$+
\sum_{\substack{{r_1,r_2,r_3=1}\\r_1=r_2\ne r_3}}^{n}\frac{a_{r_1}b_{r_1}c_{r_3}}{(r_1+k+1)^2(r_3+k+1)}+
\sum_{\substack{{r_1,r_2,r_3=1}\\r_1=r_3\ne r_2}}^{n}\frac{a_{r_1}b_{r_2}c_{r_1}}{(r_1+k+1)^2(r_2+k+1)}
$$
$$+
\sum_{\substack{{r_1,r_2,r_3=1}\\r_2=r_3\ne r_1}}^{n}\frac{a_{r_1}b_{r_2}c_{r_2}}{(r_2+k+1)^2(r_1+k+1)}+
\sum_{\substack{{r_1,r_2,r_3=1}\\r_1\ne r_2\ne r_3}}^{n}\frac{a_{r_1}b_{r_2}c_{r_3}}{(r_1+k+1)(r_2+k+1)(r_3+k+1)}.
$$
That is

$$
\sum_{r_1=1}^{n}\sum_{r_2=1}^{n}\sum_{r_3=1}^{n}
\frac{a_{r_1}b_{r_2}c_{r_3}}{(r_1+k+1)(r_2+k+1)(r_3+k+1)}  = \sum_{r_1=1}^{n}\frac{a_{r_1}b_{r_1}c_{r_1}}{(r_1+k+1)^3}
$$
$$+
\sum_{\substack{{r_1,r_2=1}\\r_1\ne r_2}}^{n}\frac{a_{r_1}b_{r_1}c_{r_2}+
a_{r_1}b_{r_2}c_{r_1}+a_{r_2}b_{r_1}c_{r_1}}{(r_1+k+1)^2(r_2+k+1)}$$
$$-
\sum_{\substack{{r_1,r_2,r_3=1}\\r_1\ne r_2\ne r_3}}^{n}\frac{a_{r_1}b_{r_2}c_{r_3}}{(r_1+k+1)(r_3-r_2)}\left(\frac{1}{r_3+k+1} -\frac{1}{r_2+k+1}\right)
.$$
We get from here

$$
\sum_{r_1=1}^{n}\sum_{r_2=1}^{n}\sum_{r_3=1}^{n}
\frac{a_{r_1}b_{r_2}c_{r_3}}{(r_1+k+1)(r_2+k+1)(r_3+k+1)}  = \sum_{r_1=1}^{n}\frac{a_{r_1}b_{r_1}c_{r_1}}{(r_1+k+1)^3}
$$
$$+
\sum_{\substack{{r_1,r_2=1}\\r_1\ne r_2}}^{n}\frac{a_{r_1}b_{r_1}c_{r_2}+
a_{r_1}b_{r_2}c_{r_1}+a_{r_2}b_{r_1}c_{r_1}}{(r_1+k+1)^2(r_2+k+1)}$$
$$-
\sum_{\substack{{r_1,r_2,r_3=1}\\r_1\ne r_3\ne r_2}}^{n}\frac{a_{r_1}b_{r_2}c_{r_3}}{(r_3-r_2)(r_2-r_1)}\left(\frac{1}{r_2+k+1} -\frac{1}{r_1+k+1}\right)
$$
\begin{equation}\label{eq:1-13}
+
\sum_{\substack{{r_1,r_2,r_3=1}\\r_1\ne r_2\ne r_3}}^{n}\frac{a_{r_1}b_{r_2}c_{r_3}}{(r_3-r_2)(r_3-r_1)}\left(\frac{1}{r_3+k+1} -\frac{1}{r_1+k+1}\right).
\end{equation}
Substituting (\ref{eq:1-12}), (\ref{eq:1-13}) into (\ref{eq:1-11}) we find:

$$
I_s = a_0b_0c_0\zeta(s) + 
\sum_{k=0}^{\infty}\sum_{r_1=1}^{n}\frac{a_0b_0c_{r_1}+
a_0c_0b_{r_1}+b_0c_0a_{r_1}}{(k+1)^{s-1}(r_1+k+1)} 
$$
$$
+ \sum_{k=0}^{\infty}\sum_{r_1=1}^{n}
\frac{a_0b_{r_1}c_{r_1}+b_0a_{r_1}c_{r_1}+c_0a_{r_1}b_{r_1}}{(k+1)^{s-2}(r_1+k+1)^2}
$$
$$-
\sum_{k=0}^{\infty}
\sum_{\substack{{r_1,r_2=1}\\r_1\ne r_2}}^{n}\frac{a_0b_{r_1}c_{r_2}+b_0a_{r_2}c_{r_1}+c_0a_{r_1}b_{r_2}}
{(k+1)^{s-2}(r_2-r_1)}\left(\frac{1}{r_2+k+1} -\frac{1}{r_1+k+1}\right) 
$$
$$
+ \sum_{k=0}^{\infty}\sum_{r_1=1}^{n}\frac{a_{r_1}b_{r_1}c_{r_1}}
{(k+1)^{s-3}(r_1+k+1)^3}+
\sum_{k=0}^{\infty}\sum_{\substack{{r_1,r_2=1}\\r_1\ne r_2}}^{n}\frac{a_{r_1}b_{r_1}c_{r_2}+a_{r_1}b_{r_2}c_{r_1}
+a_{r_2}b_{r_1}c_{r_1}}{(k+1)^{s-3}(r_1+k+1)^2(r_2-r_1)}
$$
$$+ \sum_{k=0}^{\infty}\sum_{\substack{{r_1,r_2=1}\\r_1\ne r_2}}^{n}\frac{a_{r_1}b_{r_1}c_{r_2}
+a_{r_1}b_{r_2}c_{r_1}+a_{r_2}b_{r_1}c_{r_1}}{(k+1)^{s-3}(r_2-r_1)^2}\left(\frac{1}{r_2+k+1} -\frac{1}{r_1+k+1}\right)
$$
$$+ \sum_{k=0}^{\infty}
\sum_{\substack{{r_1,r_2,r_3=1}\\r_1\ne r_2\ne r_3}}^{n}\frac{a_{r_1}b_{r_2}c_{r_3}}
{(k+1)^{s-3}}\left(\frac{1}
{(r_3-r_2)(r_3-r_1)(r_3+k+1)}\right.
$$
\begin{equation}\label{eq:1-14}
\left. + \frac{1}{(r_2-r_3)(r_2-r_1)(r_2+k+1)} + \frac{1}{(r_1-r_3)(r_1-r_2)(r_1+k+1)}\right).
\end{equation}
From (\ref{eq:1-14}) we have with $s=3$:

$$
I_3 = a_{0}b_{0}c_{0}\zeta(3) + 
\sum_{k=0}^{\infty}\sum_{r_1=1}^{n}\frac{a_0b_0c_{r_1}
+a_0c_0b_{r_1}+b_0c_0a_{r_1}}{(k+1)^{2}(r_1+k+1)}$$
$$+
\sum_{k=0}^{\infty}\sum_{r_1=1}^{n}
\frac{a_0b_{r_1}c_{r_1}+b_0a_{r_1}c_{r_1}+c_0a_{r_1}b_{r_1}}{(k+1)(r_1+k+1)^2}
$$
$$-
\sum_{k=0}^{\infty}
\sum_{\substack{{r_1,r_2=1}\\r_1\ne r_2}}^{n}\frac{a_0b_{r_1}c_{r_2}+b_0a_{r_2}c_{r_1}
+c_0a_{r_1}b_{r_2}}
{(k+1)(r_2-r_1)}\left(\frac{1}{r_2+k+1} -\frac{1}{r_1+k+1}\right) 
$$
$$
+ \sum_{k=0}^{\infty}\sum_{r_1=1}^{n}
\frac{a_{r_1}b_{r_1}c_{r_1}}
{(r_1+k+1)^3}+
 \sum_{k=0}^{\infty}\sum_{\substack{{r_1,r_2=1}\\r_1\ne r_2}}^{n}\frac{a_{r_1}b_{r_1}c_{r_2}
 +a_{r_1}c_{r_1}b_{r_2}+
 b_{r_1}c_{r_1}a_{r_2}}{(r_1+k+1)^2(r_2-r_1)}
$$
$$+ \sum_{k=0}^{\infty}\sum_{\substack{{r_1,r_2=1}\\r_1\ne r_2}}^{n}\frac{a_{r_1}b_{r_1}c_{r_2}
+a_{r_1}c_{r_1}b_{r_2}+b_{r_1}c_{r_1}a_{r_2}}{(r_2-r_1)^2}\left(\frac{1}{r_2+k+1}- \frac{1}{r_1+k+1}\right)$$
$$
+ \sum_{k=0}^{\infty}
\sum_{\substack{{r_1,r_2,r_3=1}\\r_1\ne r_2\ne r_3}}^{n}\frac{a_{r_1}b_{r_2}c_{r_3}}{(r_3-r_2)(r_3-r_1)}
\left(\frac{1}{r_3+k+1} -\frac{1}{r_1+k+1}\right)
$$
\begin{equation}\label{eq:1-16}
+\sum_{k=0}^{\infty}
\sum_{\substack{{r_1,r_2,r_3=1}\\r_1\ne r_2\ne r_3}}^{n}\frac{a_{r_1}b_{r_2}c_{r_3}}{(r_3-r_2)(r_1-r_2)}
\left(\frac{1}{r_2+k+1}-\frac{1}{r_1+k+1}\right).
\end{equation}
Let's introduce auxiliary notations:

$$ S_{00r_1} = a_0b_0c_{r_1}+a_0c_0b_{r_1}+b_0c_0a_{r_1}, \  \ S_{0r_1r_1} =
a_0b_{r_1}c_{r_1}+b_0a_{r_1}c_{r_1}+c_0a_{r_1}b_{r_1},$$
$$
S_{0r_1r_2} = a_0b_{r_1}c_{r_2}+b_0a_{r_2}c_{r_1}+c_0a_{r_1}b_{r_2},
\ \ S_{r_1r_1r_2} = a_{r_1}b_{r_1}c_{r_2}+a_{r_1}c_{r_1}b_{r_2}+b_{r_1}c_{r_1}a_{r_2}.
$$
Transforming the sums from (\ref{eq:1-16}) with use of lemma \ref{lm:2}, we find that:
\newline
1) for the second term on the right-hand side of (\ref{eq:1-16}) the following relation holds: 

\begin{equation}\label{eq:1-17}
\sum_{k=0}^{\infty}\sum_{r_1=1}^{n}\frac{S_{00r_1}}{(k+1)^{2}(r_1+k+1)}= \zeta(2)\sum_{r_1=1}^n\frac{S_{00r_1}}{r_1} -
\sum_{r_1=1}^n\frac{S_{00r_1}H_{r_1}}{r_1^2};
\end{equation}
\newline
2) for the third term of the same part of (\ref{eq:1-16}) the following equality holds:

\begin{equation}\label{eq:1-18}
\sum_{k=0}^{\infty}\sum_{r_1=1}^{n}\frac{S_{0r_1r_1}}{(k+1)(r_1+k+1)^2}=\sum_{r_1=1}^n\frac{S_{0r_1r_1}H_{r_1}}{r_1^2} -
\sum_{r_1=1}^n\frac{S_{0r_1r_1}}{r_1}\left(\zeta(2)-H^{(2)}_{r_1}\right);
\end{equation}
\newline
3) besides, we have:
\begin{equation}\label{eq:1-19}
\sum_{k=0}^{\infty}\sum_{r_1=1}^{n}
\frac{a_{r_1}b_{r_1}c_{r_1}}{(r_1+k+1)^3} = \sum_{r_1=1}^{n}a_{r_1}b_{r_1}c_{r_1}\left(\zeta(3)-H^{(3)}_{r_1}\right);
\end{equation}
\newline
4) as well we find:
$$
\sum_{k=0}^{\infty}\sum_{\substack{{r_1,r_2=1}\\r_1\ne r_2}}^{n}\frac{S_{0r_1r_2}}
{r_2-r_1}\left(\frac{1}{(k+1)(r_2+k+1)} -\frac{1}{(k+1)(r_1+k+1)}\right)
$$
\begin{equation}\label{eq:1-20}
= \sum_{\substack{{r_1,r_2=1}\\r_1\ne r_2}}^{n}\frac{S_{0r_1r_2}}
{r_2-r_1}\left(\frac{H_{r_2}}{r_2} -\frac{H_{r_1}}{r_1}\right);
\end{equation}
\newline
5) and finally we notice, that
\begin{equation}\label{eq:1-21}
\sum_{k=0}^{\infty}\sum_{\substack{{r_1,r_2=1}\\r_1\ne r_2}}^{n}\frac{S_{r_1r_1r_2}}
{(r_1+k+1)^2(r_2-r_1)}=
\sum_{\substack{{r_1,r_2=1}\\r_1\ne r_2}}^{n}\frac{S_{r_1r_1r_2}}
{r_2-r_1}\left(\zeta(2) -H^{(2)}_{r_1}\right).
\end{equation}
Substituting (\ref{eq:1-17})--(\ref{eq:1-21}) into (\ref{eq:1-16}) and doing the obvious transformations, we get

$$
I_3= \zeta(3)\sum_{r_1=0}^{n}a_{r_1}b_{r_1}c_{r_1} + \zeta(2)\left(\sum_{\substack{{r_1,r_2=1}\\r_1\ne r_2}}^{n}\left(\frac{S_{00r_1}-S_{0r_1r_1}}{r_1} + 
\frac{S_{r_1r_1r_2}}
{r_2-r_1}\right)\right)
$$
$$+ \sum_{\substack{{r_1,r_2,r_3=1}\\r_1\ne r_2\ne r_3}}^{n}\left(-a_{r_1}b_{r_1}c_{r_1}H^{(3)}_{r_1}+\left(\frac{S_{0r_1r_1}}{r_1} - 
\frac{S_{r_1r_1r_2}}{r_2-r_1}\right)H^{(2)}_{r_1}+\frac{S_{0r_1r_1}-S_{00r_1}}{r_1^2} H_{r_1}\right.$$
$$-
\frac{S_{0r_1r_2}}{r_2-r_1}\left(\frac{H_{r_2}}{r_2} -\frac{H_{r_1}}{r_1}\right)- \frac{S_{r_1r_1r_2}}{(r_2-r_1)^2}\left(H_{r_2}-H_{r_1}\right)-a_{r_1}b_{r_2}c_{r_3}*
$$
\begin{equation}\label{eq:1-23}
\left.
\left(\frac{H_{r_1}}{(r_1-r_3)(r_1-r_2)}+ \frac{H_{r_2}}{(r_2-r_1)(r_2-r_3)}+ \frac{H_{r_3}}{(r_3-r_1)(r_3-r_2)}\right)\right).
\end{equation}
where $H^{m}_{r_{\xi}}, \ \xi = 1, 2, 3; m= 1, 2, 3$, are harmonic numbers, that is sums, defined by (\ref{eq:1-1}).
Similarly, we find for $s=4$ from (\ref{eq:1-14}), that

$$
I_4 = a_{0}b_{0}c_{0}\zeta(4) + 
\sum_{k=0}^{\infty}\sum_{r_1=1}^{n}\frac{S_{00r_1}}{(k+1)^{3}(r_1+k+1)} +
\sum_{k=0}^{\infty}\sum_{r_1=1}^{n}
\frac{S_{0r_1r_1}}{(k+1)^2(r_1+k+1)^2}
$$
$$-
\sum_{k=0}^{\infty}
\sum_{\substack{{r_1,r_2=1}\\r_1\ne r_2}}^{n}\frac{S_{0r_1r_2}}
{(k+1)^2(r_2-r_1)}\left(\frac{1}{r_2+k+1} -\frac{1}{r_1+k+1}\right) 
$$
$$
+ \sum_{k=0}^{\infty}\sum_{r_1=1}^{n}
\frac{a_{r_1}b_{r_1}c_{r_1}}
{(k+1)(r_1+k+1)^3}+
 \sum_{k=0}^{\infty}\sum_{\substack{{r_1,r_2=1}\\r_1\ne r_2}}^{n}\frac{S_{r_1r_1r_2}}{(k+1)(r_1+k+1)^2(r_2-r_1)}
$$
$$+ \sum_{k=0}^{\infty}\sum_{\substack{{r_1,r_2=1}\\r_1\ne r_2}}^{n}\frac{S_{r_1r_1r_2}}{(r_2-r_1)^2(k+1)}\left(\frac{1}{r_2+k+1} - \frac{1}{r_1+k+1}\right)
$$
$$+ \sum_{k=0}^{\infty}
\sum_{\substack{{r_1,r_2,r_3=1}\\r_1\ne r_2\ne r_3}}^{n}\frac{a_{r_1}b_{r_2}c_{r_3}}{k+1}\left(\frac{1}{(r_3-r_2)(r_3-r_1)(r_3+k+1)}\right.
$$
\begin{equation}\label{eq:1-24}\left. +
\frac{1}{(r_3-r_2)(r_1-r_2)(r_2+k+1)} +
\frac{1}{(r_1-r_3)(r_1-r_2)(r_1+k+1)}\right).
\end{equation}
Transforming the sums from (\ref{eq:1-24}) on the basis of the lemma \ref{lm:2}, we get:

$$
I_4 = a_{0}b_{0}c_{0}\zeta(4) + \zeta(3)\sum_{r_1=1}^n\frac{S_{00r_1}-a_{r_1}b_{r_1}c_{r_1}}{r_1}+\zeta(2)*
$$
$$
 \sum_{\substack{{r_1,r_2=1}\\r_1\ne r_2}}^{n}\left(\frac{2S_{0r_1r_1}-S_{00r_1}-a_{r_1}b_{r_1}c_{r_1}}{r_1^2}- \frac{S_{0r_1r_2}}{r_2-r_1}\left(\frac{1}{r_2}-\frac{1}{r_1}\right)+ \frac{S_{r_1r_1r_2}}{r_1(r_2-r_1)}\right)$$
$$
+ \sum_{r_1=1}^n\frac{S_{00r_1}-2S_{0r_1r_1}
+a_{r_1}b_{r_1}c_{r_1}}{r_1^3}H_{r_1}+
\sum_{r_1=1}^n\frac{a_{r_1}b_{r_1}c_{r_1}-S_{0r_1r_1}}{r_1^2}H^{(2)}_{r_1}
$$ 
$$
+ \sum_{r_1=1}^n\frac{a_{r_1}b_{r_1}c_{r_1}}{r_1}H^{(3)}_{r_1}+ \sum_{\substack{{r_1,r_2=1}\\r_1\ne r_2}}^{n}\frac{S_{0r_1r_2}}
{r_2-r_1}\left(\frac{H_{r_2}}{r_2^2} -\frac{H_{r_1}}{r_1^2}\right)$$
$$+
\sum_{\substack{{r_1,r_2=1}\\r_1\ne r_2}}^{n}\frac{S_{r_1r_1r_2}}
{(r_2-r_1)r_1^2}H_{r_1}+ \sum_{\substack{{r_1,r_2=1}\\r_1\ne r_2}}^{n}\frac{S_{r_1r_1r_2}}
{(r_2-r_1)r_1}H^{(2)}_{r_1}$$
$$+
\sum_{\substack{{r_1,r_2=1}\\r_1\ne r_2}}^{n}\frac{S_{r_1r_1r_2}}
{(r_2-r_1)^2}\left(\frac{H_{r_2}}{r_2} -\frac{H_{r_1}}{r_1}\right)+ \sum_{\substack{{r_1,r_2,r_3=1}\\r_1\ne r_2\ne r_3}}^{n}a_{r_1}b_{r_2}c_{r_3}*
$$
\begin{equation}\label{eq:1-26}
\left. 
\left(\frac{H_{r_1}}{r_1(r_1-r_2)(r_1-r_3)}+ \frac{H_{r_2}}{r_2(r_2-r_1)(r_2-r_3)}+ \frac{H_{r_3}}{r_3(r_3-r_1)(r_3-r_2)}\right)\right).
\end{equation}
For any $s\geq 5$ we transform the terms (\ref{eq:1-14}) on the basis of the lemma  \ref{lm:2}. Carrying out algebraic calculations similar to the above ones, we get 
for  $s\geq 5$:

$$
I_s = a_{0}b_{0}c_{0}\zeta(s) + 
\left(\sum_{r_1=1}^{n}
\frac{S_{00r_1}}{r_1}\right)\zeta(s-1)$$
$$ +
\left(\sum_{\substack{{r_1,r_2=1}\\r_1\ne r_2}}^{n}\frac{S_{0r_1r_1}-S_{00r_1}}{r_1^2} -\frac{S_{0r_1r_2}}
{r_2-r_1}\left(\frac{1}{r_2} -\frac{1}{r_1}\right)\right)\zeta(s-2) 
$$
$$
+ \sum_{j=3}^{s-4}(-1)^{j-1}\left(\sum_{\substack{{r_1,r_2,r_3=1}\\r_1\ne r_2\ne r_3}}^{n}\left(\frac{S_{00r_1}-(j-1)S_{0r_1r_1}
+\frac{(j-2)(j-1)}{2}a_{r_1}b_{r_1}c_{r_1}}{r_1^{j}}\right.\right.
$$
$$+
\frac{S_{r_1r_1r_2}}{r_2-r_1}\left(\frac{j-2}{r_1^{j-1}}+\frac{1}{r_2-r_1}\left(\frac{1}{r_2^{j-2}}-\frac{1}{r_1^{j-2}}\right)\right)- 
\frac{S_{0r_1r_2}}
{r_2-r_1}\left(\frac{1}{r_2^{j-1}} -\frac{1}{r_1^{j-1}}\right) $$
$$+
a_{r_1}b_{r_2}c_{r_3}
\left(\frac{1}{(r_3-r_2)(r_3-r_1)r_3^{j-2}}+ \frac{1}{(r_2-r_3)(r_2-r_1)r_2^{j-2}}\right.$$
$$\left.\left.\left. + 
\frac{1}{(r_1-r_3)(r_1-r_2)r_1^{j-2}}\right)\right)\right)\zeta(s-j)+
$$
$$
(-1)^{s-3}
\left(\sum_{\substack{{r_1,r_2,r_3=1}\\r_1\ne r_2\ne r_3}}^{n}\left(\frac{a_{r_1}b_{r_1}c_{r_1}\left(1-\frac{(s-5)(s-4)}{2}\right)-
S_{00r_1}+(s-4)S_{0r_1r_1}}{r_1^{s-3}}\right.\right.
$$
$$+
\frac{S_{0r_1r_2}}{r_2-r_1}\left(\frac{1}{r_2^{s-4}}-\frac{1}{r_1^{s-4}}\right)- \frac{S_{r_1r_1r_2}}{r_2-r_1}\left(\frac{s-5}{r_1^{s-4}}+\frac{1}{r_2-r_1}\left(\frac{1}{r_2^{s-5}}-\frac{1}{r_1^{s-5}}\right)\right)
$$
$$
- a_{r_1}b_{r_2}c_{r_3}
\left(\frac{1}{(r_3-r_2)(r_3-r_1)r_3^{s-5}}+ 
\frac{1}{(r_2-r_3)(r_2-r_1)r_2^{s-5}}\right.$$
$$\left.\left.\left. +
\frac{1}{(r_1-r_3)(r_1-r_2)r_1^{s-5}}\right)\right)\right)\zeta(3)
$$
$$
+ (-1)^{s-2}\left(\sum_{\substack{{r_1,r_2,r_3=1}\\r_1\ne r_2\ne r_3}}^{n}\left(\frac{(s-2)S_{0r_1r_1}
-S_{00r_1}-\frac{(s-2)(s-3)}{2}a_{r_1}b_{r_1}c_{r_1}}{r_1^{s-2}}\right.\right.
$$
$$+
\frac{S_{0r_1r_2}}{r_2-r_1}\left(\frac{1}{r_2^{s-3}}-\frac{1}{r_1^{s-3}}\right)-
\frac{S_{r_1r_1r_2}}{r_2-r_1}
\left(\frac{s-3}{r_1^{s-3}}+
\frac{1}{r_2-r_1}
\left(\frac{1}{r_2^{s-4}}-\frac{1}{r_1^{s-4}}\right)\right)$$
$$+
a_{r_1}b_{r_2}c_{r_3}
\left(\frac{1}{(r_3-r_2)(r_3-r_1)r_3^{s-4}}+ \frac{1}{(r_2-r_3)(r_2-r_1)r_2^{s-4}}\right.$$
$$\left.\left.\left.\left.+ \frac{1}{(r_1-r_3)(r_1-r_2)r_1^{s-4}}\right)\right)\right)\right)\zeta(2)
$$
$$
+ (-1)^{s-2}\sum_{\substack{{r_1,r_2,r_3=1}\\r_1\ne r_2\ne r_3}}^{n}\left(\left(\frac{-(s-2)S_{0r_1r_1}+
S_{00r_1}+\frac{(s-2)(s-3)}{2}a_{r_1}b_{r_1}c_{r_1}}{r_1^{s-1}}\right.\right.$$
$$\left.+
\frac{(s-3)S_{r_1r_1r_2}}{(r_2-r_1)r_1^{s-2}}\right)H_{r_1}-
\frac{S_{0r_1r_2}}{r_2-r_1}\left(\frac{H_{r_2}}{r_2^{s-2}}-\frac{H_{r_1}}{r_1^{s-2}}\right)+$$
$$
\frac{S_{r_1r_1r_2}}{(r_2-r_1)^2}\left(\frac{H_{r_2}}{r_2^{s-3}}-\frac{H_{r_1}}{r_1^{s-3}}\right)
+ a_{r_1}b_{r_2}c_{r_3}\left(\frac{H_{r_1}}
{(r_1-r_2)(r_1-r_3)r_1^{s-3}}\right.$$
$$\left.\left.
+ \frac{H_{r_2}}
{(r_2-r_1)(r_2-r_3)r_2^{s-3}}+
\frac{H_{r_3}}
{(r_3-r_2)(r_3-r_1)r_3^{s-3}}\right)\right)$$
$$
+ (-1)^{s-3}\sum_{\substack{{r_1,r_2=1}\\r_1\ne r_2}}^{n}\left(\frac{S_{0r_1r_1}
-(s-3)a_{r_1}b_{r_1}c_{r_1}}
{r_1^{s-2}}
- \frac{S_{r_1r_1r_2}}{(r_2-r_1)r_1^{s-3}}\right)H^{(2)}_{r_1}
$$
\begin{equation}\label{eq:1-34}
 +
(-1)^{s-4}\sum_{r_1=1}^{n}a_{r_1}b_{r_1}c_{r_1}
\frac{H^{(3)}_{r_1}}{r_1^{s-3}}.
\end{equation}
Formulae (\ref{eq:1-23}), (\ref{eq:1-26}), (\ref{eq:1-34}) are representable in a more convenient for observation and computation form, grouping the factors for the same zeta-constants and replacing the sums with the non-coinciding indices by double and triple sums. Thus we have proved the following theorem:

\begin{theorem}\label{th:1}
Let $P_n(x),$ $Q_n(x)$ and $T_n(x)$ be three polynomials of degree $n,$ $n\geq 1;$

\begin{equation}\label{eq:1-35}
P_n(x) = a_0 + a_1 x + \dots + a_n x^n,
\end{equation}
\begin{equation}\label{eq:1-36}
Q_n(x) = b_0 + b_1 x + \dots + b_n x^n;
\end{equation}
\begin{equation}\label{eq:1-37}
T_n(x) = c_0 + c_1 x + \dots + c_n x^n;
\end{equation}
$a_0, a_1, \dots, a_n, \ b_0, b_1, \dots, b_n, \ c_0, c_1, \dots, c_n$ --
arbitrary numbers. Define the integral $I_s$, \ $s\geq 3$, as 

\begin{equation}\label{eq:1-38}
I_s = I_s(n) = \int_0^1\dots\int_0^1\frac{P_n(x_1)Q_n(x_2)T_n(x_3)}{1-x_1x_2x_3\dots x_s}dx_1 dx_2dx_3\dots dx_s.
\end{equation}
Then the following relations hold:

\begin{align}
\begin{split}
I_3 &= A_{s-2,3}\zeta(3) - A_{s-2,2}\zeta(2) - A_{s-2},\\ 
I_4 &= A_{s-3,4}\zeta(4) + A_{s-3,3}\zeta(3) - 
A_{s-3,2}\zeta(2) - A_{s-3},\\ 
\dots &\dots\dots\dots\dots\dots\dots\dots\dots\dots\dots\dots\dots\dots\dots\dots\dots\dots  \\
I_s &= A_{1,s}\zeta(s) + A_{1,s-1}\zeta(s-1) + \dots +  A_{1,3}\zeta(3) - A_{1,2}\zeta(2) - A_{1},\\
\end{split}
\label{eq:1-39}
\end{align}
%
where for $s= 5, 6, 7, \dots$,

\begin{equation}\label{eq:1-40}
A_{s-2,3} = \sum_{r=1}^{n}a_{r}b_{r}c_{r},  \  \  \  A_{s-2,2} = 
\sum_{r=1}^{n}\sum_{l=0}^{r-1}\frac{S_{rrl}- S_{llr}}{r-l} 
\end{equation}
$$
A_{s-2} = \sum_{r=1}^{n}a_{r}b_{r}c_{r}H^{(3)}_{r} - \sum_{r=1}^{n}\sum_{l=0}^{r-1}(S_{rrl}- S_{llr})\left(\frac{H^{(2)}_{r}-H^{(2)}_{l}}{r-l} + \frac{H_{r}-H_{l}}{(r-l)^2}\right)$$
\begin{equation}\label{eq:1-41}
+
\sum_{r=2}^{n}\sum_{l=1}^{r-1}\sum_{i=0}^{l-1}
\left(S_{irl}+ S_{ilr}\right)
\left(\frac{H_{i}}{(i-r)(i-l)}+
\frac{H_{r}}{(r-i)(r-l)} + \frac{H_{l}}{(l-i)(l-r)}\right),
\end{equation}
\begin{equation}\label{eq:1-42}
A_{s-3,4} = a_{0}b_{0}c_{0},  \  \  \  A_{s-3,3} = 
\sum_{r=1}^{n}\frac{S_{00r}- a_{r}b_{r}c_{r}}{r}, 
\end{equation}
%
%
%
$$
A_{s-3,2} = \sum_{r=1}^{n}\frac{a_{r}b_{r}c_{r}+S_{00r}-2S_{0rr}}{r^2}+
$$
\begin{equation}\label{eq:1-43}
 \sum_{r=2}^{n}\sum_{l=1}^{r-1}\left(\frac{S_{0rl}+ S_{0lr}}{r-l}
\left(\frac{1}{r} - \frac{1}{l}\right) -\frac{1}{r-l}\left(\frac{S_{rrl}}{r}- 
\frac{S_{llr}}{l}\right)\right),
\end{equation}
$$
A_{s-3} = \sum_{r=1}^{n}\left(\frac{2 S_{0rr} -a_rb_rc_r-S_{00r}}{r^3}H_{r}+
\frac{S_{0rr}- a_{r}b_{r}c_{r}}{r^2}H^{(2)}_{r} -
\frac{a_{r}b_{r}c_{r}}{r}H^{(3)}_{r}\right)$$
$$- \sum_{r=2}^{n}\sum_{l=1}^{r-1}\left(\frac{S_{rrl}- S_{llr}}{(r-l)^2}
\left(\frac{H_{r}}{r} - \frac{H_{l}}{l}\right) +\frac{1}{r-l}\left(S_{rrl}\frac{H_{r}}{r^2}- 
S_{llr}\frac{H_{l}}{l^2}\right)\right.$$
$$\left.+\frac{1}{r-l}\left(S_{rrl}\frac{H^{(2)}_{r}}{r}- 
S_{llr}\frac{H^{(2)}_{l}}{l}\right)\right)-
\sum_{r=2}^{n}\sum_{l=1}^{r-1}\sum_{i=0}^{l-1}
(S_{irl}+ S_{ilr})\left(\frac{H_{i}}{i(i-r)(i-l)} \right.
$$
\begin{equation}\label{eq:1-44}
\left. + \frac{H_{r}}{r(r-i)(r-l)} + \frac{H_{l}}{l(l-i)(l-r)}\right),
\end{equation}
\begin{equation}\label{eq:1-45}
A_{1,s} = a_{0}b_{0}c_{0},  \  \  \  A_{1,s-1} = 
\sum_{r=1}^{n}\frac{S_{00r}}{r}, 
\end{equation}
\begin{equation}\label{eq:1-46}
A_{1,s-2} = 
\sum_{r=1}^{n}\frac{S_{0rr}-S_{00r}}{r^2} - 
\sum_{r=2}^{n}\sum_{l=1}^{r-1}\frac{S_{0lr}+S_{0rl}}{r-l}\left(\frac{1}{r}-\frac{1}{l}\right), 
\end{equation}
for $j= 3, 4,\dots, s-2:$
$$
A_{1,s-j} = (-1)^{j-1}\left(\sum_{r=1}^{n}\frac{S_{00r}-(j-1)S_{0rr} +\frac{(j-2)(j-1)}{2}a_{r}b_{r}c_{r}}{r^j}\right.$$
$$-
\sum_{r=2}^{n}\sum_{l=1}^{r-1}\left(\frac{S_{llr}- S_{rrl}}{(r-l)^2}
\left(\frac{1}{r^{j-2}} - \frac{1}{l^{j-2}}\right)+
\frac{j-2}{r-l}\left(\frac{S_{llr}}{l^{j-1}}- \frac{ 
S_{rrl}}{r^{j-1}}\right)\right)
$$
%
$$
- \sum_{r=2}^{n}\sum_{l=1}^{r-1}\sum_{i=0}^{l-1}
(S_{irl}+ S_{ilr})
\left(\frac{H_{i}}{i^{j-2}(i-r)(i-l)}\right.
$$
\begin{equation}\label{eq:1-47}\left.\left.
+\frac{H_{r}}{r^{j-2}(r-i)(r-l)} + \frac{H_{l}}{l^{j-2}(l-i)(l-r)}\right)\right),
\end{equation}
$$
A_{1} = (-1)^{s-1}\left(\sum_{r=1}^{n}
\left(\frac{-S_{00r}+(s-2)S_{0rr} -\frac{(s-2)(s-3)}{2}a_rb_rc_r}{r^{s-1}}H_{r}\right.\right.$$
$$\left. +
\frac{(s-3)a_{r}b_{r}c_{r}- S_{0rr}}{r^{s-2}}H^{(2)}_{r}+
\frac{a_{r}b_{r}c_{r}}{r^{s-3}}H^{(3)}_{r}\right)$$
$$-
\left(\sum_{r=2}^{n}\sum_{l=1}^{r-1}
\frac{1}{r-l}\left(\frac{H^{(2)}_{l}S_{llr}}{l^{s-3}}- \frac{H^{(2)}_{r}
S_{rrl}}{r^{s-3}}\right)\right.
$$
$$+\left.
\frac{s-3}{r-l}\left(\frac{H_{l}S_{llr}}{l^{s-2}}- 
\frac{H_{r}S_{rrl}}{r^{s-2}}\right)+\frac{S_{llr}- S_{rrl}}{(r-l)^2}
\left(\frac{H_{r}}{r^{s-3}} - \frac{H_{l}}{l^{s-3}}\right)\right)
$$
%
$$
-
\sum_{r=2}^{n}\sum_{l=1}^{r-1}\sum_{i=0}^{l-1}
(S_{irl}+ S_{ilr})\left(\frac{H_{i}}{i^{s-3}(i-r)(i-l)}\right.
$$
\begin{equation}\label{eq:1-48}\left.\left.
+ \frac{H_{r}}{r^{s-3}(r-i)(r-l)} + \frac{H_{l}}{l^{s-3}(l-i)(l-r)}\right)\right),
\end{equation}
where for $0\leq\mu\leq n;$ $0\leq\nu\leq n;$ $0\leq\lambda\leq n;$

\begin{equation}\label{eq:1-49}
S_{\mu\nu\lambda} = a_{\mu}b_{\nu}c_{\lambda} + b_{\mu}c_{\nu}a_{\lambda} + c_{\mu}a_{\nu}b_{\lambda},
\end{equation}
$H_n^{(m)}$ are harmonic numbers defined by (\ref{eq:1-1}).

\end{theorem}

\section{New formulas for approximation of zeta-constants}

Note, the theorem \ref{th:1} formulas have the most general form, the coefficients $a_0, a_1, \dots, a_n$; $ b_0, b_1, \dots, b_n$; $c_0, c_1, \dots, c_n$ are arbitrary numbers.

To approximate effectively a zeta-constant
$\zeta(s)$, \ $s\geq 3$, on the basis of formulas
(\ref{eq:1-39})--(\ref{eq:1-49}), one should select such polynomials (\ref{eq:1-35})--(\ref{eq:1-37}) which could provide a smallness of the absolute value of integral (\ref{eq:1-38}), for example of the order of $2^{-Cn}$, \ $C=const\geq 1$.
We assume that polynomials (\ref{eq:1-35}), (\ref{eq:1-36}), (\ref{eq:1-37}) are chosen  in such a way that $I_s$ is small, 

\begin{equation}\label{eq:1-50}
I_s = I_s(n)=\theta_s(n) =\theta_s.
\end{equation}
Then we have the system:

\[
\begin{array}{ccccc}
A_{1, s}\zeta(s)+A_{1, s-1}\zeta(s-1)&+\dots+A_{1, 3}\zeta(3)&=&A_{1, 2}\zeta(2)+ A_1 +\theta_s\\
A_{2, s-1}\zeta(s-1)&+\dots+A_{2, 3}\zeta(3)&=&A_{2, 2}\zeta(2)+ A_2+\theta_{s-1}\\
\dots\dots\dots\dots&\dots\dots\dots\dots&\dots\dots\dots\dots&\dots\dots\dots\dots\\
&A_{s-2, 3}\zeta(3)&=&A_{s-2, 2}\zeta(2)+ A_{s-2}+\theta_{3}\\
\end{array}
,\]
where the coefficients $A_{i,j}$ are defined by (\ref{eq:1-40})--(\ref{eq:1-49}) for $s-2$ values $I_r$, $r= 3, 4, \dots, s.$ 
The value $\zeta(s)$ is expressed by the ratio of determinants:

$$
\Delta = A_{1, s}A_{2, s-1}\dots A_{s-2, 3}, \  \  \  \  \zeta(s) = \frac{\Delta_s}{\Delta},
$$
wherein

\[\mathbf{\Delta_s}=\begin{vmatrix}
A_{1, 2}\zeta(2)+ A_1+\theta_s & A_{1, s-1} & \dots & \dots & \dots & A_{1, 3}\\
A_{2, 2}\zeta(2)+ A_2+\theta_{s-1} & A_{2, s-1} & \dots & \dots & \dots & A_{2, 3}\\
\dots & \dots & \dots & \dots & \dots & \dots\\
A_{s-2, 2}\zeta(2)+ A_{s-2}+\theta_3 & 0 & \dots & \dots & \dots & A_{s-2, 3}\\
\end{vmatrix} =
\mathbf{\Delta_s^{(1)}+\Delta_s^{(2)}+\Delta_s^{(3)}},
\]
where

\[\mathbf{\Delta_s^{(1)}}=\begin{vmatrix}
A_{1, 2}\zeta(2) & A_{1, s-1} & \dots  & A_{1, 3}\\
A_{2, 2}\zeta(2) & A_{2, s-1} & \dots  & A_{2, 3}\\
\dots & \dots & \dots &  \dots \\
A_{s-2, 2}\zeta(2) & 0 & \dots &  A_{s-2, 3}\\
\end{vmatrix}
\]

\[\mathbf{\Delta_s^{(2)}}=\begin{vmatrix}
A_1 & A_{1, s-1} & \dots  & A_{1, 3}\\
A_2 & A_{2, s-1} & \dots  & A_{2, 3}\\
\dots & \dots & \dots &  \dots \\
A_{s-2} & 0 & \dots &  A_{s-2, 3}\\
\end{vmatrix}, \  \  \
\mathbf{\Delta_s^{(3)}}=\begin{vmatrix}
\theta_s & A_{1, s-1} & \dots  & A_{1, 3}\\
\theta_{s-1} & A_{2, s-1} &  \dots & A_{2, 3}\\
\dots & \dots & \dots &  \dots\\
\theta_{3} & 0 &  \dots & A_{s-2, 3}\\
\end{vmatrix}.
\]
Thus,

$$
\zeta(s) = \frac{\Delta_s}{\Delta} = \frac{1}{\Delta}\sum_{\nu=1}^{s-2}
\left(A_{\nu, 2}\zeta(2) + A_{\nu} + \theta_{s+1-\nu}\right)\Delta_{\nu s},
$$
where $\Delta_{\nu s}$ --- the corresponding algebraic complements,
i.e.

\begin{equation}\label{eq:1-51}
\zeta(s) = \frac{\zeta(2)}{\Delta}\sum_{\nu=1}^{s-2}A_{\nu, 2}\Delta_{\nu s} + \frac{1}{\Delta}\sum_{\nu=1}^{s-2}A_{\nu}\Delta_{\nu s} + \frac{1}{\Delta}\sum_{\nu=1}^{s-2}\theta_{s+1-\nu}\Delta_{\nu s}.
\end{equation}

\section{Approximation}

Now we estimate the values of $\theta_{s},$ $s\geq 3$, from (\ref{eq:1-50}), (\ref{eq:1-51}) for especially chosen polynomials.
Like in \cite{10}, as main polynomials providing the approximation accuracy we take two polynomials (\ref{eq:0-1}): the shifted Legendre polynomial as $P_n(x)$, the binomial polynomial as $Q_n(x)$.
The third polynomial $T_n(x)$ will be written in the canonical form. We have

\begin{lemma}\label{lm:3}
Let

\begin{equation}\label{eq:1-60}
P_n(x) = \frac{1}{n!}\left(\frac{d}{dx}\right)^n\left(x^n(1-x)^n\right)= a_0 + a_1 x + \dots + a_n x^n,
\end{equation}

\begin{equation}\label{eq:1-61}
a_r = \frac{(-1)^r(n+r)!}{(r!)^2(n-r)!},
\end{equation}

\begin{equation}\label{eq:1-62}
Q_n(x) = (1-x)^n =b_0 + b_1 x + \dots + b_n x^n;
\end{equation}

\begin{equation}\label{eq:1-63}
b_r = (-1)^r\frac{n!}{r!(n-r)!}.
\end{equation}
Let
\begin{equation}\label{eq:1-64}
T_n(x) = c_0 + c_1 x + \dots + c_n x^n; \   \   \   c^* = \max_{0\leq i\leq n}|c_i|;
\end{equation}
Then for the integral

$$
I_s = I_s(n) = \int_0^1\dots\int_0^1\frac{P_n(x_1)Q_n(x_2)T_n(x_3)}{1-x_1x_2x_3\dots x_s}dx_1 dx_2 dx_3\dots dx_s, \  \  s\geq 3,
$$
the following estimate is valid
\begin{equation}\label{eq:1-66}
|I_s| \leq \frac{c^*}{2^{2n}}.
\end{equation}
\end{lemma}
\begin{proof}
From (\ref{eq:1-62}), (\ref{eq:1-63}) we have for $k\geq 1$ 

$$
\int_0^1 x^kQ_n(x)dx = \int_0^1 x^k(1-x)^n dx = B(k+1, n+1) = \frac{k!n!}{(k+n+1)!},
$$
where $B(x,y)$ --- the Euler beta function  (see, for example, \cite{22}, \cite{23}), $B(x,y)=\frac{\Gamma(x)\Gamma(y)}{\Gamma(x+y)},$ where $\Gamma(x+1)= x\Gamma(x)$ is the Euler gamma function.
At the same time, from (\ref{eq:1-60}), (\ref{eq:1-61}) we find by integration by parts:

$$
\int_0^1 x^kP_n(x)dx = \frac{1}{n!}\int_0^1 x^k\frac{d}{dx}\left(\frac{d^{n-1}}{dx^{n-1}}\left(x^n(1-x)^n\right)\right)dx =
$$
$$= (-1)^n\frac{k(k-1)\dots(k-n+1)}{n!}B(k+1,n+1)= (-1)^n\frac{k(k-1)\dots(k-n+1)}{(k+1)\dots(k+n)(k+n+1)},
$$
for $k\geq n$. If $k<n$, then (see \cite{10} for details)

$$
\int_0^1 x^kP_n(x)dx = 0.
$$
Thus, we have with the selected
$P_n(x)$ and $Q_n(x)$:

$$
\int_0^1 x^kQ_n(x)dx = B(k+1, n+1), \  \  \  \  \  \  \  \  \  \  \  \  \  \  \  \  \ \ \ \ \ \ \ \ \ \ \ \ \ \ \ \ \ \ \ \ \ \ \ \ \ \ \ \ \ \ \     
$$

\[
\int_0^1 x^kP_n(x)dx =\begin{cases}
0, &\text{if $k\leq n-1$}\\
(-1)^n\frac{k(k-1)\dots(k-n+1)}{n!}B(k+1,n+1), &\text{if $k\geq n,$}
\end{cases}
\]
$$
\int_0^1 x^kT_n(x)dx = \sum_{i=0}^n\frac{c_i}{k+1+i} = T(k, n). \  \  \  \  \  \  \  \  \  \  \  \  \  \  \  \  \ \ \ \ \ \ \ \ \ \ \ \ \ \ \ \ \ \ \ \ \ \ \ \ \ \ \      
$$
Consequently

$$
I_s = I_s(n) = \int_0^1\dots\int_0^1\frac{P_n(x_1)Q_n(x_2)T_n(x_3)}{1-x_1x_2x_3\dots x_s}dx_1 dx_2dx_3\dots dx_s=
$$
$$
= \sum_{k=0}^{\infty}\int_0^1\dots\int_0^1(x_1x_2\dots x_s)^kP_n(x_1)Q_n(x_2)T_n(x_3)dx_1 dx_2 dx_3\dots dx_s =
$$
\begin{equation}\label{eq:1-68}
=(-1)^n\sum_{k=n}^{\infty}\frac{k(k-1)\dots(k-n+1)}{n!}\frac{B^2(k+1, n+1)T(k, n)}{(k+1)^{s-3}}.
\end{equation}
We find from (\ref{eq:1-68}):

$$
\left|I_s\right|\leq c^*\sum_{j=0}^{\infty}\frac{(n+j)!}{j!n!}\frac{B^2(n+1+j, n+1)}{(n+1+j)^{s-3}}
$$
\begin{equation}\label{eq:1-69}
\leq c^*\sum_{j=0}^{\infty}\frac{B(n+1+j, n+1)}{(2n+j+1)(n+1+j)^{s-3}}\frac{\Gamma^2(n+j+1)}{\Gamma(2n+j+1)\Gamma(j+1)}.
\end{equation}
Since (see, for example, \cite{22})

$$
\frac{\Gamma(\alpha)\Gamma(\beta)}{\Gamma(\alpha+\gamma)\Gamma(\beta-\gamma)}=
\prod_{\nu=0}^{\infty}\left(1+\frac{\gamma}{\alpha+\nu}\right)\left(1-\frac{\gamma}{\beta+\nu}\right),
$$
then

$$
\frac{\Gamma^2(n+j+1)}{\Gamma(2n+j+1)\Gamma(j+1)}=
\prod_{\nu=0}^{\infty}\left(1-\left(\frac{n}{n+1+j+\nu}\right)^2\right) < 1
$$
for any $n\geq 1$.  Consequently  
from (\ref{eq:1-69}) for any $s\geq 3$ follows the estimate

$$
|I_s| \leq c^*\sum_{j=0}^{\infty}\frac{B(n+1+j, n+1)}{(2n+j+1)(n+1+j)^{s-3}} \leq \frac{c^*}{2}\sum_{j=0}^{\infty}B(n+1+j, n+1).$$
Taking into account (see \cite{23}) that $\sum_{j=0}^{\infty}B(x, y+j) = B(x-1, y)$
and corresponding estimates for beta functions (see, for example, \cite{22}, see also \cite{10}), we find from here
$$
|I_s|\leq \frac{c^*}{2}B(n, n+1)\leq \frac{c^*}{4}B(n, n)\leq
\frac{c^*}{4}
2^{1-2n}B\left(\frac 12, n\right)\leq c^*{2^{-2n}},
$$
that is the estimate (\ref{eq:1-66}). 
\end{proof}

\end{document}